\documentclass[10pt]{article}

\usepackage{amsfonts,amssymb,amsthm}

\usepackage{makeidx,graphics}
\usepackage{epsfig}
\usepackage{amsmath}
\usepackage{multicol}
\usepackage{amscd}
\usepackage{longtable}
\DeclareGraphicsExtensions{bmp,eps}
\usepackage{hyperref} 
\usepackage{graphics}
\usepackage{graphicx}
\usepackage[applemac]{inputenc}

\usepackage{color}

\hypersetup{
     colorlinks   = true,
     linkcolor    = cyan,
     citecolor    = blue,
     urlcolor = red 
}

\def\be{\begin{equation}}
\def\ee{\end{equation}}
\def\bea{\begin{eqnarray}}
\def\eea{\end{eqnarray}}

\def\1{\'{\i}}                           

  \def\>#1{{\mathbf#1}}

\definecolor{verdeoscuro}{cmyk}{1,0,0.6,0.5}

\newtheorem{theorem}{Theorem}

\begin{document}

\thispagestyle{empty}

\hfill \today

\ 
\vspace{0.5cm}

\begin{center}

\LARGE{{

Integrable deformations of Rikitake systems, Lie bialgebras and bi-Hamiltonian structures}}

\end{center}

\begin{center} {\sc Angel Ballesteros$^1$, Alfonso Blasco$^1$, Ivan Gutierrez-Sagredo$^{2,3}$}
\end{center}

\begin{center} {$^1$\it{Departamento de F\1sica,  Universidad de Burgos, 
09001 Burgos, Spain} \\
$^2$\it{Departamento de Matem\'aticas y Computaci\'on,  Universidad de Burgos, 
09001 Burgos, Spain} \\
$^3$\it{Departamento de Matem\'aticas, Estad\'istica e Investigaci\'on Operativa,  Universidad de La Laguna, Spain}
}

e-mail: angelb@ubu.es, ablasco@ubu.es, igsagredo@ubu.es
\end{center}

  \medskip

\begin{abstract} 
Integrable deformations of a class of Rikitake dynamical systems are constructed by deforming their underlying Lie-Poisson Hamiltonian structures, which are considered linearizations of Poisson--Lie structures on certain (dual) Lie groups. By taking into account that there exists a one-to one correspondence between Poisson--Lie groups and Lie bialgebra structures, a number of deformed Poisson coalgebras can be obtained, which allow the construction of integrable deformations of coupled Rikitake systems. Moreover, the integrals of the motion for these coupled systems can be explicitly obtained by means of the deformed coproduct map. The same procedure can be also applied when the initial system is bi-Hamiltonian with respect to two different Lie-Poisson algebras. In this case, to preserve a bi-Hamiltonian structure under deformation, a common Lie bialgebra structure for the two Lie-Poisson structures has to be found. Coupled dynamical systems arising from this bi-Hamiltonian deformation scheme are also presented, and the use of collective `cluster variables', turns out to be enlightening in order to analyse their dynamical behaviour.  As a general feature, the approach here presented provides a novel connection between Lie bialgebras and integrable dynamical systems. 

\end{abstract}


\bigskip




\tableofcontents

\section{Introduction}

It is well-known that for  any Lie algebra $(\mathfrak{g},[\,\cdot,\cdot\,]_{\mathfrak{g}})$ we can consider the dual vector space $\mathfrak g^*$ and the Poisson algebra of smooth functions
$C^\infty(\mathfrak g^*)$ endowed with the Poisson bracket
\be
\left\{ f, g \right\}(\xi)= \left\langle \xi, \left[df, dg \right] \right\rangle, \qquad \xi \in \mathfrak g^*, \quad f,g \in C^\infty(\mathfrak g^*),
\label{eq:LiePoisson11}
\ee
where $\langle \cdot, \cdot \rangle$ is the pairing between $\mathfrak{g}$ and $\mathfrak g^*$. 
This Lie--Poisson bracket is given by a set of fundamental Poisson brackets which are just the Poisson analogues of the Lie brackets defining $\mathfrak{g}$.

Let us now consider a generic $N$-dimensional (ND) dynamical system defined by the Lie-Poisson structure associated to a given finite-dimensional Lie algebra $\mathfrak{g}$, together with a Hamiltonian function $\cal{H}\in C^\infty(\mathfrak g^*)$. The Hamilton equations of motion of the system will be given by the system of ODEs 
\be
\dot x_i=\left\{ x_i, \cal{H} \right\},
\qquad
i=1,\dots, N  ,
\ee
where $x_i$ are coordinate functions on $\mathfrak g^*$. By construction, the set of Casimir functions ${\cal C}_i$ of the Lie-Poisson structure provides a set of integrals of the motion in involution for such a system, and level sets of the Casimir functions define the symplectic leaves of the Lie-Poisson structure \eqref{eq:LiePoisson11}.

The problem of finding integrable deformations of the system admits many solutions. A first obvious family of integrable deformations can be obtained by modifying the Hamiltonian $\cal{H}$ and by preserving the Lie-Poisson algebra \eqref{eq:LiePoisson11} and its Casimir functions ${\cal C}_i$. Indeed, more interesting integrable deformations will be the ones obtained by considering nontrivial deformations of the underlying Lie-Poisson structure together with the associated deformed versions of the Casimir functions.

In the three-dimensional (3D) case the latter construction is straightforward, since we recall that the algebra of smooth functions on $(x,y,z)$ endowed with the bracket
\be
\{x,y\}=f\,\frac{\partial F}{\partial z}, \qquad
\{y,z\}=f\,\frac{\partial F}{\partial x}, \qquad
\{z,x\}=f\,\frac{\partial F}{\partial y},
\label{3dpois}
\ee
define always a Poisson structure for any choice of the smooth functions $f(x,y,z)$ and $F(x,y,z)$, and the Casimir function for the Poisson structure is just $F(x,y,z)$~\cite{Fokas}. Therefore if we consider an initial 3D Hamiltonian system defined by a Hamiltonian function $\mathcal{H}$ and a pair of functions $(f,F)$, any deformed Casimir function $F_\eta(x,y,z)$ such that $\lim_{\eta\to 0} F_\eta(x,y,z) = F(x,y,z)$ provides a deformation of the bracket~\eqref{3dpois} that, by construction, will generate an integrable deformation of the system defined by $\mathcal{H}$. Indeed, $\mathcal{H}$ could be further modified in terms of new parameters, and this secondary deformation (in terms of the Hamiltonian) would  superpose to the previous one by preserving integrability.

The aim of this paper is to show that, among this infinite zoo of possible integrable deformations, there exists a 
distinguished and very restricted subclass of deformations that preserves the additional Poisson coalgebra symmetry that exists for any Lie-Poisson dynamical system, which was introduced in~\cite{BR} and is based on the existence of a coproduct map that endows the Lie-Poisson algebra with a Poisson coalgebra structure. In this framework, the integrable deformation arises as a consequence of the existence of a compatible deformed version of the coproduct map, which is just the group law of a dual non-abelian Poisson-Lie group. Moreover, the essential property of such Poisson coalgebra deformations is the fact that the deformed coproduct map provides coupled pairs of deformations of the initial system in a systematic and constructive way, which was introduced in~\cite{BR, BRcluster,jpcs, loopsAIP}, latter applied to 3D Lotka-Volterra systems in~\cite{LVpla} and further generalised to some bi-Hamiltonian systems in~\cite{JDE16,BMR}. In this framework, the deformed coproduct map of the deformed Casimir functions will give rise to the integrals of the motion of the deformed system, thus providing its complete integrability structure in the Liouville sense.

In this paper we complete this approach by analysing explicitly the very restricted number of integrable deformations of a given relevant type of dynamical systems (the Rikitake ones) that are endowed with a deformed Poisson coalgebra symmetry. Through this specific example we will show that all possible Poisson coalgebra deformations for a single Lie-Poisson system can be explicitly obtained and are in one to one correspondence with the nonequivalent classes of  Lie bialgebra structures of the dual Lie algebra $\mathfrak{g}^\ast$ (in the sense of Lie bialgebra theory) that is associated to the initial Lie algebra $\mathfrak g$. This fact is based on the well-known result by Drinfel'd that states that Poisson-Lie structures on a given Lie group are in one to one correspondence with Lie bialgebra structures on its Lie algebra~\cite{Drinfeld1983hamiltonian}.
As a consequence we show that, in general, the Lie bialgebra classification problem for a given Lie algebra (which is essential in the theory of quantum deformations of Lie algebras and groups~\cite{Dri, CP}) turns out to be also relevant in dynamical systems theory. 

Moreover, through the examples here presented and by making use of the so-called `cluster variables'~\cite{BRcluster} we will get a deeper insight into the dynamical properties of the coupled integrable deformations of the Rikitake systems that will be obtained through Poisson coalgebra deformations. In particular, these variables provide a neat picture of the coupled system dynamics as a set of global `collective' variables with the same dynamics of the uncoupled system plus another set of variables which encodes the coupling arising from each Lie bialgebra structure. Also, we will illustrate the strong compatibility constraints that arise when trying to obtain Poisson coalgebra deformations that preserve the bi-Hamiltonian structure that exist for certain Lie-Poisson Rikitake systems.

The paper is organized as follows. In the next Section the class of integrable Rikitake-type systems is reviewed, together with the Lie-Poisson Hamiltonian structures that underly their integrability. Specifically, we will select the three types of integrable Rikitake systems that will be considered in the paper (called A, B and AB). In Section 3, the interpretation of Poisson-Lie groups as (coalgebra) deformations of Lie-Poisson algebras is presented, and the construction will be illustrated through the specific case of the Lie-Poisson (1+1) Poincar\'e algebra, which will be the relevant one in the rest of the paper. In Section 4, the Poincar\'e Poisson coalgebra symmetry of an integrable Rikitake system of type A will be presented. Furthermore, we prove that bi-Hamiltonian deformations with Poisson coalgebra symmetry does not exist for the Rikitake system of type B. Finally, a Poisson-Lie group structure suitable for the construction of a bi-Hamiltonian deformation for the Rikitake AB system will be presented. In Section 5, we compute the integrable deformations of the Rikitake system arising from the Poisson-Lie structures constructed in Section 4. In Section 6, the coupled systems coming from the Poisson-Lie symmetry for the Rikitake AB system will be explicitly obtained. Finally, a section including some comments and open problems closes the paper.

\section{Lie-Poisson algebras for Rikitake systems}
\label{sec:LiePoisson}
\setcounter{equation}{0}

The Rikitake system~\cite{Rikitake}, which describes the dynamics of  two connected identical frictionless disk dynamos, is a well-known model of the (aperiodic) Earth's geomagnetic field reversals. Following~\cite{Valls}, we consider the generalized Rikitake dynamical system given by
 \begin{equation}
\begin{array}{lll}
\dot{x}=-\mu x+y (z + \beta_1), &\quad 
\dot{y}=-\mu y+ x(z-\beta_2), & \quad 
\dot{z}=\alpha-x y ,
\end{array}\label{Rikigen}
\end{equation}
where $\mu,\beta_1,\beta_2$ and $\alpha$ are real parameters (the original Rikitake system~\cite{Rikitake} is the one with $\beta_2=0$).
The integrability properties of the system~\eqref{Rikigen} have been thouroughly studied~\cite{Tudoran, llibre1, llibre2} (see also~\cite{LB, Lazureanu, Brasilenos, GuptaYadav, RTudoran} and references therein). Moreover, (integrable) deformations of the Rikitake system have attracted a considerable attention (see \cite{LazureanuBinzar,Lazureanu2,Huang,Lazureanu3}). 

There exist two distinguished families of integrable cases for this dynamical system (A and B), where $\mu=0$ arises always as a common integrability condition. In the sequel we describe both integrable families together with their associated Lie-Poisson algebras, and we sketch the ways in which integrable deformations for them can be constructed.

\subsection{Case A}

When $\alpha\neq 0$ and $\beta_1=\beta_2=\mu=0$ the system 
 \begin{equation}
\begin{array}{lll}
\dot{x}= y \,z , &\quad 
\dot{y}=x\,z, & \quad 
\dot{z}=\alpha-x y ,
\end{array}\label{Rikicaso1}
\end{equation}
has a Lie--Poisson Hamiltonian structure provided by the Hamiltonian function
\be
\mathcal{H}^A=\dfrac{1}{2}(x^{2}+z^{2})-\alpha\,\log (x+y),
\label{ham1}
\ee
and the Poisson structure given by the Lie-Poisson (1+1) Poincar\'e algebra~\cite{Tudoran, LB}
\be
\{x,y\}=0, \qquad
\{x,z\}=y, \qquad
\{y,z\}=x .
\label{alg1}
\ee
Therefore, the dynamical system~\eqref{Rikicaso1} is recovered as Hamilton's equations, namely
\be
\begin{array}{llll}
\dot{x}=\{x,\mathcal{H}^A\}=y\, z, &\quad
\dot{y}=\{y,\mathcal{H}^A\}=x\, z, &\quad
\dot{z}=\{z,\mathcal{H}^A\}=\alpha-x\, y.
\end{array}\label{IR}
\ee
Both the Hamiltonian~\eqref{ham1} and the Casimir function of the Lie-Poisson (1+1) Poincar\'e algebra given by
\begin{equation}
\mathcal C^A=\frac14 \left( y^{2}-x^{2}	\right) 
\label{eq:casA}
\end{equation}
are integrals of the motion for this system. 

The orbits of \eqref{Rikicaso1} are contained in the level sets of the Casimir function \eqref{eq:casA}, $\mathcal C^A = k$. For any fixed value $k>0$, the simplectic realization
\be
\begin{array}{lll}
x=2 \sqrt{k}\sinh q, & y= 2 \sqrt{k} \cosh q, & z=p.
\end{array}
\label{eq:symplecAbook}
\ee
defines a map from the whole Poisson manifold to a given simplectic leaf, and allows us to write the Hamiltonian \eqref{ham1} in terms of the canonical variables $(q,p)$, namely
\begin{equation}
\mathcal H^A = \frac{p^2}{2} + 2 k \sinh^2 q - \alpha \left( \log \left( 2 \sqrt k \right) + q \right) \, ,
\end{equation}
which is a natural Hamiltonian system. Therefore, for a given energy $\mathcal H^A  = E$, the dynamics of the system can be written as
\begin{equation}
t-t_0 = \int_{q_0}^q \frac{\mathrm d s}{\sqrt{2 \left( E - 2 k \sinh^2 s +2 (\log(2 \sqrt k) + s)\right)}} .
\end{equation}

\subsection{Case B}
 When $\alpha=\mu=0$ the system is Liouville integrable provided that  $\beta_1=\beta_2=\beta \neq 0$ \cite{llibre1, llibre2}, namely
 \begin{equation}
\begin{array}{lll}
\dot{x}=y (z + \beta), &\quad 
\dot{y}= x(z-\beta), & \quad 
\dot{z}=-x y ,
\end{array}\label{Rikicaso2}
\end{equation}
and this is the case studied in~\cite{Lazureanu}. Moreover, its integrability is provided by a bi-Hamiltonian structure in terms of two four-dimensional (hereafter 4D) Poisson--Lie algebras: the centrally extended (1+1) Poincar\'e algebra and the centrally extended $so(3)$ algebra. 

Explicitly, the first Hamiltonian structure of the system~\eqref{Rikicaso2} is given by the non-trivial central extension of the Poincar\'e Lie-Poisson algebra given by
\be
\{x,y\}_0=2\,\beta\,I ,\qquad 
\{y,z\}_0=x, \qquad
\{z,x\}_0=-y, \qquad
\{ I ,\cdot \}_0=0, 
\label{extendedP}
\ee
together with the Hamiltonian
\be
\mathcal{H}_0=\frac14 (x^2 + y^2 + 2\,z^2).
\ee
The second constant of the motion is given by the quadratic Casimir of~\eqref{extendedP} (note that $I$ is also a trivial Casimir function), namely
\be
\mathcal{C}_0=-\frac14 \left(\frac{x^2 - y^2}{I} +4\,\beta\,z\right).
\ee

The second Hamiltonian structure for~\eqref{Rikicaso2} is given by the 4D centrally extended $so(3)$  Lie--Poisson algebra
\be
\{x,y\}_1=2\,z, \qquad
\{y,z\}_1=x, \qquad
\{z,x\}_1=y, \qquad
\{ I ,\cdot \}_1=0, 
\label{extendedso3}
\ee
together with the Hamiltonian
\be
\mathcal{H}_1=-\frac14 (x^2 - y^2 +4\,\beta\,z).
\ee
Again, the second non-trivial constant of the motion is given by the quadratic Casimir of~\eqref{extendedso3}, namely
\be
\mathcal{C}_1=\frac14 (x^2 + y^2 + 2\,z^2).
\ee

As it is usual in dynamical systems endowed with bi-Hamiltonian structures, $\mathcal{H}_0\equiv \mathcal{C}_1$ and  $\mathcal{H}_1\equiv \mathcal{C}_0$. Moreover, the linear Poisson structures $\{\cdot, \cdot\}_0$ and $\{\cdot, \cdot\}_1$ on $\mathbb{R}^4$ turn out to be compatible, in the sense that we can define a one-parametric family of Lie-Poisson structures (a Poisson pencil)
\begin{equation}
\quad\{.,.\}_{\lambda}=(1-\lambda)\{.,.\}_{0}+\lambda\{.,.\}_{1},\quad \mbox{ with } \lambda \in \mathbb{R} ,
\end{equation}
whose explicit Poisson brackets are given by
\be
\{x,y\}_\lambda=(1-\lambda)\,2\,\beta\,I +\lambda\,2\,z
\,\qquad
\{y,z\}_\lambda=x
, \qquad
\{z,x\}_\lambda=(-1+2\,\lambda)\, y
, \qquad
\{ I ,\cdot \}_\lambda=0.
\label{pencilb}
\ee
Therefore, the approach introduced in~\cite{BMR} (see also \cite{GPMPR}) shows that for each common 1-cocycle for the 4D Lie algebras~\eqref{extendedP} and~\eqref{extendedso3} we could, in principle, construct an integrable bi-Hamiltonian deformation of the system. However, in this particular case, we will show in Section 4.2 that such a common cocycle does not exist. Nevertheless, the usual integrable deformation procedure that we apply to the case A can be also used in this case, thus obtaining a different deformation for each of the Lie--Poisson structures.

\subsection{Case AB}

Nevertheless, when exploring the possible common cocycles for the Poisson pencil~\eqref{pencilb} we realize that the case $\beta=0$ does admit such a bi-Hamiltonian structure with a common cocycle (note that this is also Case A with $\alpha=0$). This system, namely
 \begin{equation}
\begin{array}{lll}
\dot{x}=y \,z , &\quad 
\dot{y}= x\,z , & \quad 
\dot{z}=-x \, y ,
\end{array}\label{Rikicaso3}
\end{equation}
admits the bi-Hamiltonian description 
 \begin{equation}
\begin{array}{lll}
\dot{x}= \{ x, \mathcal H_0 \}_0 = \{ x, \mathcal H_1 \}_1 , &\quad 
\dot{y}= \{ y, \mathcal H_0 \}_0 = \{ y, \mathcal H_1 \}_1 , & \quad 
\dot{z}= \{ z, \mathcal H_0 \}_0 = \{ z, \mathcal H_1 \}_1 ,
\end{array}
\end{equation}
where the first Hamiltonian function and Poisson brackets read 
\be
\mathcal{H}_{0}^{AB}=\dfrac{1}{4}(x^2+y^2+2 z^2),
\label{eq:H_AB0}
\ee
\be
\begin{array}{lll}
\{x,y\}_{0}=0, & \{x,z\}_{0}=y, & \{y,z\}_{0}=x,
\end{array}\label{Poinc}
\ee
while the second ones are given by
\be
\mathcal{H}_{1}^{AB}=\dfrac{1}{4}(y^2-x^2),
\label{eq:H_AB1}
\ee
\be
\begin{array}{lll}
\{x,y\}_{1}=2z, & \{x,z\}_{1}=-y, & \{y,z\}_{1}=x .
\end{array}\label{so3}
\ee
Then, by following~\cite{BMR} we will construct a bi-Hamiltonian deformation of \eqref{Rikicaso3} by means of the Poisson pencil
\be
\{x,y\}_\lambda=2\,\lambda\,z
\,\qquad
\{y,z\}_\lambda=x
, \qquad
\{z,x\}_\lambda=(-1+2\,\lambda)\, y\, .
\label{pencil2}
\ee
Also, the integrable coupled systems arising from the coalgebra approach will be explicitly presented. Note that the Casimir function for (\ref{pencil2}) has the following expression:
\be
\mathcal{C}_{\lambda}=\dfrac{x^2+(2\,\lambda-1)y^2+2\,\lambda\,z^2}{8\,\lambda-4} .
\ee

For any $\lambda \neq \frac12$, we have the following one-parameter family of simplectic realizations 
\begin{equation}
x=\sqrt{4\,k+2\,\lambda\, p^2}\,\sinh(\sqrt{1-2\lambda}\,q), \qquad y=\sqrt{\dfrac{4\,k+2\,\lambda\,p^2}{1-2\,\lambda}}\,\cosh(\sqrt{1-2\lambda}\,q), \qquad z=p .
\label{eq:simplectic_real_AB}
\end{equation}
This allows us to formally compute the trajectories of the system in two different ways. For the first one, corresponding to $\lambda = 0$ in \eqref{eq:simplectic_real_AB}, we have that the Hamiltonian \eqref{eq:H_AB0} reads
\begin{equation}
\mathcal{H}_{0}^{AB} = \frac{p^2}{2} + k \cosh(2 q) \, ,
\end{equation}
which has again the shape of a natural Hamiltonian system. When $\mathcal{H}_{0}^{AB} = E$, its trajectories will be given by
\begin{equation}
t-t_0 = \int_{q_0}^q \frac{\mathrm d s}{\sqrt{2 \left( E - k \cosh (2 s) \right)}} .
\end{equation}
For the second one, corresponding to setting $\lambda = 1$ in \eqref{eq:simplectic_real_AB}, the symplectic realization of the Hamiltonian \eqref{eq:H_AB1} is given by
\begin{equation}
\mathcal{H}_{1}^{AB} = - \cos(2 q) \left( \frac{p^2}{2} + k \right) .
\end{equation}
For fixed energies $\mathcal{H}_{1}^{AB} = E$, the formal solution would be given by
\begin{equation}
t-t_0 = \int_{q_0}^q \frac{\mathrm d s}{\sqrt{-2 \cos (2 s) \left( E + k \cos (2 s) \right)}} .
\end{equation}

\section{Lie bialgebras: the $(1+1)$ Poincar\'e case}
\setcounter{equation}{0}

A well-known result by Drinfel'd \cite{Drinfeld1983hamiltonian} ensures that Poisson-Lie (PL) structures on a simply connected Lie group $G$ are in one-to-one correspondence with Lie bialgebra structures $\left(\mathfrak{g},\delta\right)$ on $\mathfrak{g}=\mbox{Lie}\left(G\right)$, where the linearization of the PL structure in terms of the local coordinates on $G$ is a Lie algebra whose dual is the map $\delta$. 
More explicitly, a Lie bialgebra $(\mathfrak{g},\delta)$ is a Lie algebra $\mathfrak{g}$ with structure tensor $c^k_{ij}$
\be
[X_i,X_j]=c^k_{ij}X_k ,
\label{liealg}
\ee
together with a skewsymmetric cocommutator map
$
\delta:{\mathfrak{g}}\to {\mathfrak{g}}\otimes {\mathfrak{g}}
$
fulfilling the two following conditions:
\begin{itemize}
\item i) $\delta$ is a 1-cocycle, {\em  i.e.},
$$
\delta([X,Y])=[\delta(X),\,  Y\otimes 1+ 1\otimes Y] + 
[ X\otimes 1+1\otimes X,\, \delta(Y)] ,\qquad \forall \,X,Y\in
\mathfrak{g}.
\label{1cocycle}
$$
\item ii) The dual map $\delta^\ast:\mathfrak{g}^\ast\otimes \mathfrak{g}^\ast \to
\mathfrak{g}^\ast$ is a Lie bracket on $\mathfrak{g}^\ast$.
\end{itemize}
Therefore any cocommutator $\delta$ will be of the form 
\be
\delta(X_i)= f^{jk}_i\,X_j \otimes X_k \, ,
\label{precoco}
\ee
where $f^{jk}_i$ is the structure tensor of the dual Lie algebra $\mathfrak{g}^\ast$ defined by
\be
[\hat x^j,\hat x^k]=f^{jk}_i\, \hat x^i \, ,
\label{dualL}
\ee
where $\langle  \hat x^j,X_k \rangle=\delta_k^j$. Note that, in general, the Lie algebra $\mathfrak{g}^\ast$ is not isomorphic to $\mathfrak{g}$. Moreover, the dual Lie bialgebra $\mathfrak{g}^\ast$ is naturally equipped with a Lie bialgebra structure $(\mathfrak{g}^\ast,\delta^\ast)$ since the Lie algebra structure on $\mathfrak{g}$ gives rise to the dual cocommutator map
$\delta^\ast: \mathfrak{g}^\ast\rightarrow \mathfrak{g}^\ast\otimes \mathfrak{g}^\ast$, namely
\be
\delta^\ast(\hat x^k)= c_{ij}^k\,  \hat x^i \otimes \hat x^j. 
\ee
Lie bialgebras are the tangent counterpart of Poisson-Lie groups \cite{CP} in an similar way that Lie algebras are the tangent structures associated to Lie groups. In our case, the extra Poisson structure on the Lie group $G$ induces the cocommutator map $\delta$ on $T_e G \simeq \mathfrak g$. This implies that, for any dynamical system that admits a Hamiltonian formulation in terms of a linear Poisson structure (Lie-Poisson structure), each possible bialgebra structure for the associated Lie algebra gives rise to a deformed Hamiltonian system. The remarkable properties of deformed systems so constructed stem from the fact that the algebra of smooth functions $\mathcal C^\infty(G)$ on the Poisson-Lie group $G$ is automatically endowed with a Poisson-Hopf algebra structure, characterized by the coproduct map $\Delta : \mathcal C^\infty(G) \to \mathcal C^\infty(G) \otimes \mathcal C^\infty(G)$. This implies that deformed systems constructed in this way are very rigid, compared to arbitrary deformations of the Poisson structure and/or the Hamiltonian. However, this Poisson-Hopf structure automatically guarantees that canonically coupled systems can be constructed by means of the coproduct map $\Delta$. For a detailed explanation of all these results, see~\cite{JDE16,BMR} and references therein.

Let us now consider the $(1+1)$-dimensional Poincar\'e Lie algebra
\be
\begin{array}{lll}
[X,Y]=0, & [X,Z]=Y, & [Y,Z]=X,
\end{array}
\label{eq:PoincareLie}
\ee
which is relevant for Case A, Case B (when the central extension is considered) and Case AB. The full classification of nonisomorphic Lie bialgebra structures for this algebra is given in~\cite{Gomez2000} and leads to seven non-isomorphic families of cocommutator maps. Out of them, we will consider the trivial one whose dual Lie algebra is Abelian, and two non-trivial cocommutators, one with nilpotent and one with solvable dual Lie algebras. Explicitly:

\begin{itemize}

\item The trivial Lie bialgebra structure $\delta(X)=\delta(Y)=\delta(Z)=0$ with dual Abelian Lie algebra.

\item The Lie bialgebra
\be
\delta(X)= \eta X\wedge Z,
\qquad
\delta(Y)= \eta Y\wedge Z,
\qquad
\delta(Z)=0,
\label{eq:bookdelta}
\ee
with dual Lie algebra isomorphic to the `book'\, Lie algebra \cite{BBMbook}
\be
\begin{array}{lll}
[\hat x,\hat y]=0, &
[\hat x,\hat z]=\eta\, \hat x, &
[\hat y,\hat z]=\eta\, \hat y. 
\end{array} \label{dalg1}
\ee

\item The Lie bialgebra
\be
\delta(X)=0,\qquad
\delta(Y)=0,\qquad
\delta(Z)=\eta\, X\wedge Y
\label{h3}
\ee
with dual Lie algebra isomorphic to the Heisenberg-Weyl Lie algebra, namely
\be
[\hat x,\hat y]=\eta\,\hat z, \qquad
[\hat x,\hat z]=0, \qquad
[\hat y,\hat z]=0. 
\label{dalgh3}
\ee

\end{itemize}

This means that the three dual Lie bialgebras $(\mathfrak{g}^\ast,\delta^\ast)$ will have the same dual cocommutator map $\delta^*$ arising from the Poincar\'e Lie algebra \eqref{eq:PoincareLie}, namely
\be
\delta^*(\hat x)=\hat y \wedge \hat z,\qquad
\delta^*(\hat y)=\hat x \wedge \hat z,\qquad
\delta^*(\hat z)=0. 
\label{dualP}
\ee
Therefore, the PL structures for the three dual groups will have the Poincar\'e Lie algebra as their linearization and therefore the two PL structures coming from nontrivial cocommutators can be considered as Poisson-Hopf algebra deformations of the Poincar\'e Lie-Poisson algebra.

\section{Poisson-Lie deformations and Lie bialgebras}
\setcounter{equation}{0}

In this Section we study the Poisson-Hopf algebras relevant to the different deformations of the Rikitake systems that we are considering. In particular, for Case A we explicitly compute both of the non-trivial Poisson-Hopf algebras associated to the two non-trivial $(1+1)$-dimensional Poincar\'e Lie bialgebras previously introduced and show explicitly that in the limit $\eta \to 0$ we recover the Poisson version of \eqref{eq:PoincareLie}. Afterwards, we study the possible bi-Hamiltonian deformations: we prove that Case B does not admit a bi-Hamiltonian deformation by showing that the Poisson pencil \eqref{pencilb} does not admit a common cocycle, and finally, we explicitly construct a bi-Hamiltonian deformation for Case AB.

\subsection{Case A: two non-equivalent Poisson-Hopf algebras from Lie bialgebras}

Let us start by explicitly computing the Poisson-Hopf algebras associated to the Lie bialgebras \eqref{eq:bookdelta} and \eqref{h3}, which will result in integrable deformations of the Case A system.

\subsubsection{The `book' group deformation}

The first step to be performed is the construction of the Lie group $G^*$, with $\mathfrak{g}^{*}=\text{Lie}(G^{*})$. Thus, we start by exponentiating a faithful representation of $\mathfrak{g}^{*}$ \eqref{dalg1}. 
Taking, for instance, the adjoint representation we get
\be
\begin{array}{lll}
\varrho (\hat x)=\begin{pmatrix}
0 & 0 & \eta \\
0 & 0 & 0\\
0 & 0  & 0
\end{pmatrix} &
\varrho (\hat y)=\begin{pmatrix}
0 & 0 & 0\\
0 & 0 & \eta \\
0 & 0  & 0
\end{pmatrix}&
\varrho (\hat z)=\begin{pmatrix}
-\eta & 0 & 0\\
0 & -\eta & 0 \\
0 & 0  & 0
\end{pmatrix}.
\end{array}\label{adj}
\ee
By using \eqref{adj} we can parametrize a generic element of the group $G^*$ in terms of the  $(x,y,z)$-coordinates, namely
\be
\mathcal{G}_{1}=
\exp(z\, \varrho(\hat z))  \exp (y\, \varrho(\hat y)) \exp (x\, \varrho(\hat x))=
\begin{pmatrix}
e^{-\eta\, z} & 0 & \eta\, x e^{-\eta z}\\
0 & e^{-\eta\, z} & \eta\, y e^{-\eta z} \\
0 & 0  & 1
\end{pmatrix},
\label{eq:Gbook}
\ee
and the multiplication rule between two group elements reads
\be
\mathcal{G}_{1}\cdot \mathcal{G}_{2}=
\begin{pmatrix}
e^{-\eta\, (z_1+z_2)} & 0 & \eta\, e^{-\eta\, (z_1+z_2)} ( x_1 e^{\eta z_2}+x_2)\\
0 & e^{-\eta\, (z_1+z_2)} & \eta\, e^{-\eta\, (z_1+z_2)} ( y_1 e^{\eta z_2}+y_2) \\
0 & 0  & 1
\end{pmatrix}.\label{cr}
\ee
The coproducts for the  $(x,y,z)$-coordinates are derived from the group multiplication and read \cite{JDE16,BMR}
\be
\begin{array}{l}
\Delta_{\eta}(x)=x\otimes e^{\eta\, z}+1\otimes x=x_{1}e^{\eta\, z_{2}}+x_{2} ,\\
\Delta_{\eta}(y)=y\otimes e^{\eta\, z}+1\otimes y=y_{1}e^{\eta\, z_{2}}+y_{2} , \\
\Delta_{\eta}(z)=z\otimes 1+1\otimes z=z_{1}+z_{2} ,
\end{array} \label{defcr1}
\ee
where we have identified the coordinates of $\mathcal{G}_{1}$ with the tensor space at the left side of the tensor product, and the ones of $\mathcal{G}_{2}$ at the right one. 

As the deformed Poisson bracket is always quadratic in terms of the group entries, the application of the compatibility conditions stated by the fact that the deformed coproduct relations \eqref{defcr1} have to provide a Poisson map leads to the most generic PL structure on the book group, which was explicitly obtained in \cite{BBMbook}. 

From this result, and imposing that the linearisation of such a generic PL structure has to lead to the Poisson version of \eqref{eq:PoincareLie}, a straightforward computation leads to the unique solution given by the Poisson bracket
\be
\begin{array}{lll}
\{x,y\}_{\eta}=\dfrac{\eta}{2}\left(
y^{2}-x^{2}
\right), \quad&
\{x,z\}_{\eta}=y,\quad &
\{y,z\}_{\eta}=x.
\end{array} \label{defalg2}
\ee
It is straightforward to see that in the limit $\eta \to 0$ we recover the Poisson version of the $(1+1)$-dimensional Poincar\'e Lie algebra \eqref{eq:PoincareLie}, 
\be
\begin{array}{lll}
\lim\limits_{\eta \rightarrow 0}=\{x,y\}_{\eta}=\{x,y\}=0, & \lim\limits_{\eta \rightarrow 0}=\{x,z\}_{\eta}=\{x,z\}=y, &
\lim\limits_{\eta \rightarrow 0}=\{y,z\}_{\eta}=\{y,z\}=x .
\end{array}
\ee
The Casimir of \eqref{defalg2} is given by
\be
\mathcal{C}_{\eta}=e^{-\eta\, z}\left(
y^{2}-x^{2}
\right) ,
\label{casd2}
\ee
and again
\be
\lim\limits_{\eta \rightarrow 0}\mathcal{C}_{\eta}=\mathcal{C} = y^2 - x^2, 
\ee
which coincides with \eqref{eq:casA} up to a constant factor.

\subsubsection{The Heisenberg-Weyl deformation} 

The third Lie bialgebra structure~\eqref{h3} has a dual Lie algebra $\mathfrak{g}^{*}$ given by
\be
[\hat x,\hat y]=\eta\,\hat z, \qquad
[\hat x,\hat z]=0, \qquad
[\hat y,\hat z]=0,
\label{dalgh3}
\ee
which is isomorphic to the Heisenberg-Weyl algebra. 
Following a similar procedure as in the previous case, we start with the faithful representation 
\begin{equation}
\varrho(\hat z)=\left( 
\begin{array}{ccc}
0 & 0 & \eta\\
0 & 0 & 0\\
0 & 0 & 0
\end{array}
\right),  \qquad 
\varrho(\hat x)=\left( 
\begin{array}{ccc}
0 & \eta & 0\\
0 & 0 & 0\\
0 & 0 & 0
\end{array}
\right), \qquad
\varrho(\hat y)=\left( 
\begin{array}{ccc}
0 & 0 & 0\\
0 & 0 & \eta\\
0 & 0 & 0
\end{array}
\right) ,
\end{equation}
and taking a parametrisation of the Lie group $G^*$ defined by
\begin{equation}
\mathcal{G}_{1}=
\exp(x \varrho(\hat x) ) \exp(y \varrho(\hat y)) \exp(z \varrho(\hat z) )=
\left(
\begin{array}{ccc}
1 & \eta\,y &  \eta^2\,x y+ \eta\,z\\
0 & 1 &  \eta\,x\\
0 & 0 & 1
\end{array}
\right),
\end{equation}
the coproduct map can be straightforwardly computed from the matrix product $\mathcal{G}_{1}\cdot \mathcal{G}_{2}$, and it reads
\be
\begin{array}{l}
\Delta(x)=x \otimes 1 + 1 \otimes x =x_{1}+x_{2},\\
\Delta(y)=y \otimes 1 + 1 \otimes y=y_{1}+y_{2},\\
\Delta(z)=z \otimes 1 + 1 \otimes z - \eta\, x \otimes y=z_{1}+z_{2}-\eta\, y_{1}\,x_{2}.
\end{array}\label{cdh}
\ee
The compatible Poisson-Lie structure can be straightforwardly computed, and its fundamental Poisson brackets read 
\be
\begin{array}{lll}
\{x,y\}_{\eta}= 0, \quad&
\{x,z\}_{\eta}=\eta\, x+ y ,\quad &
\{y,z\}_{\eta}= x+\eta \,y.
\end{array} 
\label{defalgh3}
\ee
The Casimir $\mathcal{C}_{\eta}$ for \eqref{defalgh3} is given by
\be
\mathcal{C}_{\eta}=(y^2-x^2)^{1-\eta}(y-x)^{2\,\eta} .
\label{eq:Casdefalgh3}
\ee
Again, it is straightforward to check that in the limit $\eta \to 0$ we recover, up to a constant, the Poisson version of the $(1+1)$-dimensional Poincar\'e Lie algebra \eqref{eq:PoincareLie} and the Casimir function \eqref{eq:casA}.

\subsection{Bi-Hamiltonian deformations}

Integrable deformations of Case B coming independently from each of its two Lie-Poisson Hamiltonian structures could be obtained by a completely analogous procedure to the one presented above, and will be based on the complete classification of Lie bialgebra structures for the centrally extended (1+1) Poincar\'e and $so(3)$ algebras. For the sake of brevity we shall not discuss them in this paper in detail. However, Rikitake system B is bi-Hamiltonian and thus the existence of the Poisson pencil \eqref{pencilb} allows in principle to construct bi-Hamiltonian integrable deformations by following the prodedure introduced in~\cite{JDE16,BMR}. The essential point is that these deformations are possible if and only if there exist common cocycles $\delta$ for the two Lie algebras underlying the bi-Hamiltonian structure, namely  the centrally extended Poincar\'e and $so(3)$ Lie algebras. In the following we prove that Case B does not admit bi-Hamiltonian integrable deformations since this condition cannot be fulfilled. Nevertheless, we will also show how Case AB does admit such a common cocycle, and we will compute and analyse a particular bi-Hamiltonian deformation for this system.

\begin{theorem}
The system 
 \begin{equation}
\begin{array}{lll}
\dot{x}=y (z + \beta), &\quad 
\dot{y}= x(z-\beta), & \quad 
\dot{z}=-x y ,
\end{array}
\label{Rikicaso2th}
\end{equation}
where $\beta \neq 0$, does not admit any bi-Hamiltonian Poisson-Lie integrable deformation. 
\end{theorem}

\begin{proof}
System \eqref{Rikicaso2th} is bi-Hamiltonian with respect to the non-trivially centrally extended $(1+1)$-dimensional Lie-Poisson Poincar\'e algebra \eqref{extendedP} and to the trivially centrally extended $so(3)$ \eqref{extendedso3} Lie-Poisson algebra. The necessary and sufficient condition for the existence of bi-Hamiltonian integrable deformations of \eqref{Rikicaso2th} is the existence of a common, \emph{i.e.} independent of $\lambda$, cocommutator map $\delta$ for the Poisson pencil 
\be
[X,Y]_\lambda=(1-\lambda)\,2\,\beta\,W +\lambda\,2\,Z
\,\qquad
[Y,Z]_\lambda=X
, \qquad
[Z,X]_\lambda=(-1+2\,\lambda)\, Y
, \qquad
[ W ,\cdot ]_\lambda=0.
\ee
We set a generic pre-cocommutator map $\delta: \mathfrak g \to \mathfrak g \wedge \mathfrak g$
\be
\begin{array}{l}
\delta (X)=\eta\,(a_{1}\,X\wedge Y+a_{2}\, X\wedge Z+a_{3}\,X\wedge W+a_{4}\,Y\wedge Z+a_{5}\,Y\wedge W+a_{6}\,Z\wedge W), \\
\delta (Y)=\eta\,(b_{1}\,X\wedge Y+b_{2}\, X\wedge Z+b_{3}\,X\wedge W+b_{4}\,Y\wedge Z+b_{5}\,Y\wedge W+b_{6}\,Z\wedge W), \\
\delta (Z)=\eta\,(c_{1}\,X\wedge Y+c_{2}\, X\wedge Z+c_{3}\,X\wedge W+c_{4}\,Y\wedge Z+c_{5}\,Y\wedge W+c_{6}\,Z\wedge W), \\
\delta (W)=\eta\,(d_{1}\,X\wedge Y+d_{2}\, X\wedge Z+d_{3}\,X\wedge W+d_{4}\,Y\wedge Z+d_{5}\,Y\wedge W+d_{6}\,Z\wedge W), \\
\end{array}
\ee
and we impose the cocycle condition
\be
\begin{array}{l}
\delta([X,Y]_{\lambda})=[1\otimes X+X\otimes 1,\delta(Y)]+[\delta(X),1\otimes Y+Y\otimes 1]=(1-\lambda)\,2\,\beta\,\delta(W) +\lambda\,2\,\delta(Z),\\
\\
\delta([Y,Z]_{\lambda})=[1\otimes Y+Y\otimes 1,\delta(Z)]+[\delta(Y),1\otimes Z+Z\otimes 1]=\delta(X),\\
\\
\delta([Z,X]_{\lambda})=[1\otimes Z+Z\otimes 1,\delta(X)]+[\delta(Z)1\otimes X+X\otimes 1]=(-1+2\lambda)\delta(Y),\\
\\
\delta([X,W]_{\lambda})=[1\otimes X+X\otimes 1,\delta(W)]+[\delta(X),1\otimes W+W \otimes 1]=0,\\
\\
\delta([Y,W]_{\lambda})=[1\otimes Y+Y\otimes 1,\delta(W)]+[\delta(Y),1\otimes W+W \otimes 1]=0,\\
\\
\delta([Z,W]_{\lambda})=[1\otimes Z+Z\otimes 1,\delta(W)]+[\delta(Z),1\otimes W+W \otimes 1]=0.
\end{array}
\ee
Solving these equations we obtain that
\be
\begin{array}{l}
\delta(X)=\eta\left( (1-2\lambda)b_{3}\, Y\wedge W +b_{4}\, X\wedge Z -b_{4}\dfrac{\beta(1-\lambda)}{\lambda} X\wedge W +(2\beta(1-\lambda)c_{2}-2\lambda c_{3}) Z\wedge W -c_{4}\, X\wedge Y \right),\\
\\
\delta(Y)=\eta\left( b_{3}\, X\wedge W +b_{4}\, Y\wedge Z -b_{4}\dfrac{\beta(\lambda-1)}{\lambda} Y\wedge W +b_{6} Z\wedge W +c_{2} X\wedge Y  \right),
\\
\\
\delta(Z)=\eta\left( c_{2} X \wedge Z +c_{3} X \wedge W +c_{4} Y \wedge Z +\left( \dfrac{(2\lambda-1)}{2\lambda}b_{6}+\dfrac{\beta(1-\lambda)}{\lambda}c_{4} \right) Y \wedge W \right),\\
\\
\delta(W)=0.
\label{eq:deltath419}
\end{array}
\ee
Since the cocommutator $\delta$ \eqref{eq:deltath419} should be defined for the complete Poisson pencil, including $\lambda=0$, we obtain that $b_{4}=0$, $b_{6}=0$, and $c_{4}=0$. Therefore we have
\be
\begin{array}{l}
\delta(X)=\eta\left(
(1-2\lambda)b_{3}\, Y \wedge W +(2\beta(1-\lambda)c_{2}-2\lambda c_{3}) Z \wedge W
\right),\\
\\
\delta(Y)=\eta\left(
b_{3}\, X \wedge W +c_{2} X \wedge Y 
\right) ,
\\
\\
\delta(Z)=\eta\left(
c_{2} X \wedge Z +c_{3} X \wedge W \right),\\
\\
\delta(W)=0.
\end{array}
\ee
The only remaining $\lambda$-independent solution is then obtained if $b_3=0$ and $c_3 = - \beta c_2$,
\begin{equation}
\delta (X)=2 \eta \beta c_2 Z \wedge W, \qquad \delta (Y)=\eta c_2 X \wedge Y, \qquad \delta (Z)=\eta c_2 (X \wedge Z - X \wedge W), \qquad \delta (W)=0.
\end{equation}
However, the co-Jacobi condition implies that the only cocommutator map is the trivial one $\delta (X)=\delta (Y)=\delta (Z)=\delta (W)=0$. Therefore, there are no non-trivial  bi-Hamiltonian Poisson-Lie integrable deformations of the system \eqref{Rikicaso2th}.

\end{proof}

\subsubsection{Bi-Hamiltonian deformation of the Rikitake AB system}

We have showed that Case A does not admit (non-trivial) bi-Hamiltonian deformations. However, when $\alpha=\mu=\beta=0$ (Case AB), the system \eqref{Rikicaso3} does admit such deformations. This case is given by the system of ODEs
\be
\begin{array}{lll}
\dot{x}= y z, &
\dot{y}=x z, &
\dot{z}=- x y .
\end{array}\label{RikiType}
\ee
This system is known to be bi-Hamiltonian respect the following two Hamiltonian structures:
\begin{enumerate}
\item[a)] First Hamiltonian structure:
The (1+1)-Poincar\'e Lie-Poisson algebra
\be
\begin{array}{lll}
\{x,y\}_{0}=0, & \{x,z\}_{0}=y, & \{y,z\}_{0}=x ,
\end{array}\label{Poinc}
\ee
where the Casimir function is given by $\mathcal{C}_{0}=\dfrac{1}{4}(y^2-x^2)$, together with the Hamiltonian function
\be
\mathcal{H}_{0}=\dfrac{1}{4}(x^2+y^2+2 z^2) .
\ee

\item[b)]  Second Hamiltonian structure:
The $\mathfrak{so}(3)$ Lie-Poisson algebra
\be
\begin{array}{lll}
\{x,y\}_{1}=2z, & \{x,z\}_{1}=-y, & \{y,z\}_{1}=x ,
\end{array}\label{so3}
\ee
with Casimir function given by $\mathcal{C}_{1}=\dfrac{1}{4}(x^2+y^2+2 z^2)$, together with the Hamiltonian 
\be
\mathcal{H}_{1}=\dfrac{1}{4}(y^2-x^2) .
\ee

\end{enumerate}

The Poisson pencil formed by these two Lie algebras reads 
\be
\begin{array}{lll}
[X,Y]_{\lambda}=2\,\lambda\,Z, &
[X,Z]_{\lambda}=(1-2\,\lambda)Y, &
[Y,Z]_{\lambda}=X .
\end{array}
\label{eq:pencil427}
\ee
It is easy to see that, in contradistinction to Case A, now there exist $\lambda$-independent common cocycles, like for instance, the one given in \eqref{eq:bookdelta},
\be
\delta(X)=\eta\,  X \,\wedge Z, \qquad \delta(Y)=\eta\,  Y\wedge Z, \qquad \delta(Z)=0 ,
\ee
whose dual Lie algebra is isomorphic to the `book' algebra \eqref{dalg1}, with Lie bracket
\be
\begin{array}{lll}
[\hat x,\hat y]=0, & [\hat x,\hat z]=\eta\, \hat x, & [\hat y,\hat z]=\eta\, \hat y .
\end{array}
\ee
The `book' Lie group can be embedded in $\mathrm{GL}(3, \mathbb R)$ as shown in \eqref{eq:Gbook}. Using this parametrisation, the coproduct $\Delta_{\eta}: \mathcal C^\infty (G^\ast) \to \mathcal C^\infty (G^\ast) \otimes \mathcal C^\infty (G^\ast)$ is given by \eqref{defcr1}. A standard computation based on \cite{BBMbook} shows that the unique Poisson structure on $G^\ast$ compatible with $\Delta_{\eta}$ with linearization given by \eqref{eq:pencil427} and therefore endowing $\mathcal C^\infty (G^\ast)$ with a Poisson-Hopf algebra is given by the fundamental Poisson brackets 
\be
\{x,y\}_{\lambda,\eta}=\dfrac{\eta^{2}[-x^2+y^2(1-2\lambda)]+2\lambda (e^{2\,\eta\,z}-1)}{2\,\eta}, \qquad
\{x,z\}_{\lambda,\eta}=(1-2\,\lambda) y, \qquad
\{y,z\}_{\lambda,\eta}=x .
\label{eq:poissonABlambda}
\ee
The Casimir function for this Poisson-Lie group structure is given by
\be
\mathcal{C}_{\lambda,\eta}=
\dfrac{e^{-\eta\,z}[\eta^{2}(-x^2+y^2(1-2\lambda))-2\lambda(e^{\eta\,z}-1)^2]}
{4\,\eta^{2}}.
\label{eq:CasimirABlambda}
\ee
The non deformed limit $\eta\rightarrow 0$ of this function is well defined
\be
\mathcal{C}_{\lambda,0}=\lim\limits_{\eta\rightarrow 0}\mathcal{C}_{\lambda,\eta}=\dfrac{1}{4}(-x^2+(1-2\lambda)y^2-2\,\lambda \, z)
\rightarrow
\begin{cases}
\lambda =0, & \mathcal{C}_{0,0}=\dfrac{1}{4}(y^2-x^2)\\
\lambda =1, & -\mathcal{C}_{1,0}=-\dfrac{1}{4}(x^2+y^2+2z^2).
\end{cases}
\ee
As expected, when $\lambda=0$ and $\lambda=1$, we obtain Poisson-Lie deformations of the Lie-Poisson (1+1)-Poincar\'e \eqref{Poinc} and $\mathfrak{so}(3)$ \eqref{so3} Lie algebras, respectively:

\begin{enumerate}

\item[a)] When $\lambda=0$, the deformed Poisson-Lie bracket becomes a deformation of the (1+1)-Poincar\'e algebra~\eqref{Poinc}, namely
\be
\begin{array}{lll}
\{x,y\}_{0,\eta}=\dfrac{\eta}{2}(y^2-x^2), &
\{x,z\}_{0,\eta}=y, &
\{y,z\}_{0,\eta}=x .
\end{array}
\label{eq:poissonABlambda0}
\ee
Note that the only deformation in the Poisson brackets is produced for the $\{x,y\}_{0,\eta}$ bracket whose $\eta\rightarrow 0$ limit vanishes. In this case the deformed Casimir \eqref{eq:CasimirABlambda} reduces to
\begin{equation}
\mathcal{C}_{0,\eta}=\dfrac{e^{-\eta\,z}}{4}(y^2-x^2),
\end{equation}
and finally, when the deformation parameter goes to zero, we recover the Casimir function $\mathcal{C}_{0} $ of the (1+1)-Poincar\'e algebra,  
\begin{equation}
\mathcal{C}_{0,0} = \lim\limits_{\eta\rightarrow 0}\mathcal{C}_{0,\eta}=\dfrac{1}{4}(y^2-x^2)=\mathcal{C}_{0} .
\end{equation}

\item[b)] When $\lambda=1$, the Poisson-Lie bracket \eqref{eq:poissonABlambda} is a deformation of the $so(3)$ algebra ~\eqref{so3}. Explicitly, we have
\be
\begin{array}{lll}
\{x,y\}_{1,\,\eta}=\dfrac{2(e^{2\,\eta\, z}-1)-\eta^{2}(x^2+y^2)}{2\,\eta}, &
\{x,z\}_{1,\,\eta}=-y, &
\{y,z\}_{1,\,\eta}=x .
\end{array}
\label{eq:poissonABlambda1}
\ee
Moreover, the Casimir function is found to be 
\be
\mathcal{C}_{1,\eta}=\dfrac{e^{-\eta\,z}(-\eta^2\,(x^2+y^2)-2(e^{\eta\,z}-1)^2)}{4\,\eta^{2}},
\ee
which, in the undeformed limit, coincides (modulo an irrelevant global sign) with the $\mathfrak{so}(3)$ Casimir function $\mathcal{C}_{1}$, namely
\be
\mathcal{C}_{1,0} = \lim\limits_{\eta\rightarrow 0}\mathcal{C}_{1,\eta}=-\dfrac{1}{4}(x^2+y^2+2z^2)=-\,\mathcal{C}_{1}.
\ee

\end{enumerate}

\section{Integrable deformations of the Rikitake system}
\setcounter{equation}{0}

In this Section we present explicitly the $\eta$-deformed dynamical systems arising from the Poisson-Lie structures obtained in the previous Section.

\subsection{Case A}

In this case, we recall that the Hamiltonian takes the form
\begin{equation}
\mathcal{H}^A=\dfrac{1}{2}(x^2+z^2)-\alpha\,\log(x+y) .
\label{eq:ham5A}
\end{equation} 
The undeformed system \eqref{IR} is obtained from the (1+1) Poincar\'e algebra \eqref{alg1}, as shown in Section \ref{sec:LiePoisson}. This system could be seen as a dynamical system defined on an Abelian (trivial cocommutator) Poisson-Lie group. Therefore $\eta$-deformations, or equivalently, non-Abelian dual Poisson-Lie structures, define integrable deformations of this system. In this case, both Poisson-Lie structures explicitly given in the previous Section define the following deformed Rikitake systems.

\paragraph{$\bullet$ The ``book'' group deformation:}

The deformed equations coming from (\ref{defalg2}) and \eqref{eq:ham5A} read 
\be
\begin{split}
\dot{x}&=y z+\dfrac{\eta}{2}(x-y)\alpha, \\
\dot{y}&=x z+\dfrac{\eta}{2}(x-y)\left(x(x+y)-\alpha\right), \\
\dot{z}&=\alpha-x y.
\end{split}
\label{eq:bookdef}
\ee
A symplectic realisation of this deformed system corresponding to a value $\kappa$ of the Casimir function \eqref{casd2} is given by
\be
\begin{array}{lll}
x=e^{-\frac{\eta\,p}{2}}\sqrt{k}\sinh q, &
y=e^{-\frac{\eta\,p}{2}}\sqrt{k}\cosh q, &
z=p.
\end{array}
\label{eq:sympbookdef}
\ee
Note that in the $\eta \to 0$ limit we recover \eqref{eq:symplecAbook}, both in the dynamical system and the symplectic realisation. Closed orbits of this deformed system are given in Figure \ref{fig:book}.

\begin{figure}[h]
    	\includegraphics[width=8cm, height=5cm]{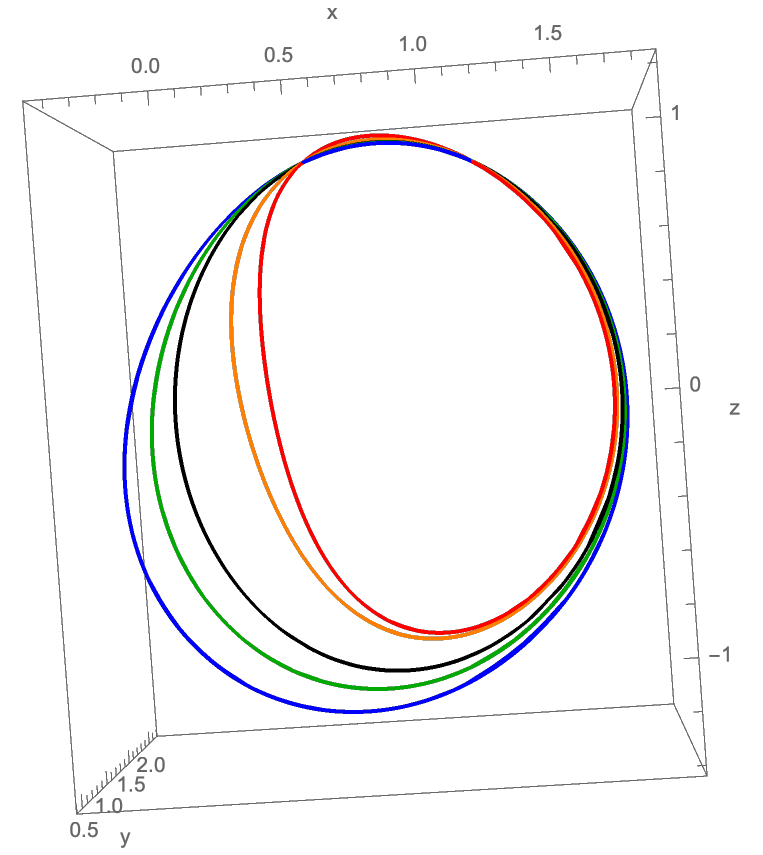}
	\includegraphics[width=8cm, height=5cm]{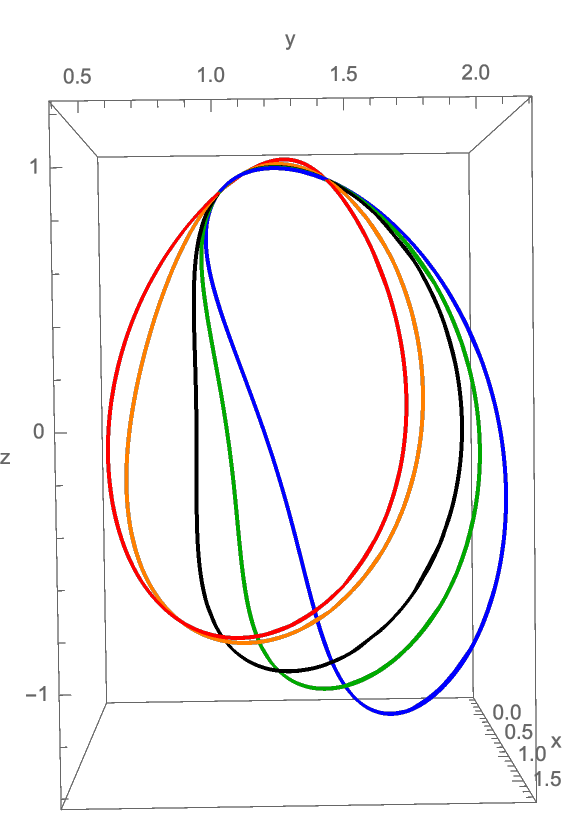}
    \caption{\label{fig:book}Projections $xz$ (left) and $yz$ (right) of the three-dimensional orbits of the deformed system \eqref{eq:bookdef} with $\eta= -0.5$  (blue), $\eta= -0.25$  (green), $\eta= 0$  (black), $\eta= 1$  (orange) and $\eta= 2$  (red). Initial conditions: $x(0)=0.5$, $y(0)=1$, $z(0)=1$.}
\end{figure}

\paragraph{$\bullet$ The Heisenberg-Weyl deformation:}

The deformed equations coming from the Poisson structure (\ref{defalgh3})  and the Hamiltonian function \eqref{eq:ham5A} are given by
\be
\begin{split}
\dot{x}&=\{x,\mathcal{H}\}_{\eta}= z\,(y+\eta\, x), \\
\dot{y}&=\{y,\mathcal{H}\}_{\eta}= z\, (x+\eta\, y), \\
\dot{z}&=\{z,\mathcal{H}\}_{\eta}=\alpha -x \, y+\eta(\alpha - x^{2}) .
\end{split}
\label{eq:heisdef}
\ee
A symplectic realisation of the Poisson algebra \eqref{defalgh3} with Casimir function \eqref{eq:Casdefalgh3} $\mathcal C_\eta = \kappa$ is given by
\be
\begin{array}{lll}
x=-e^{\eta\,q}\sqrt{k}\sinh q, & y=-e^{\eta\,q}\sqrt{k} \cosh q, & z=p.
\label{eq:sympheisdef}
\end{array}
\ee
The most significant difference between \eqref{eq:sympbookdef} and \eqref{eq:sympheisdef} is that in the former both $x$ and $y$ are functions of $q$ and $p$, while in the latter only functions of the $q$ variables appear.

\begin{figure}[h]
    	\includegraphics[width=8cm, height=5cm]{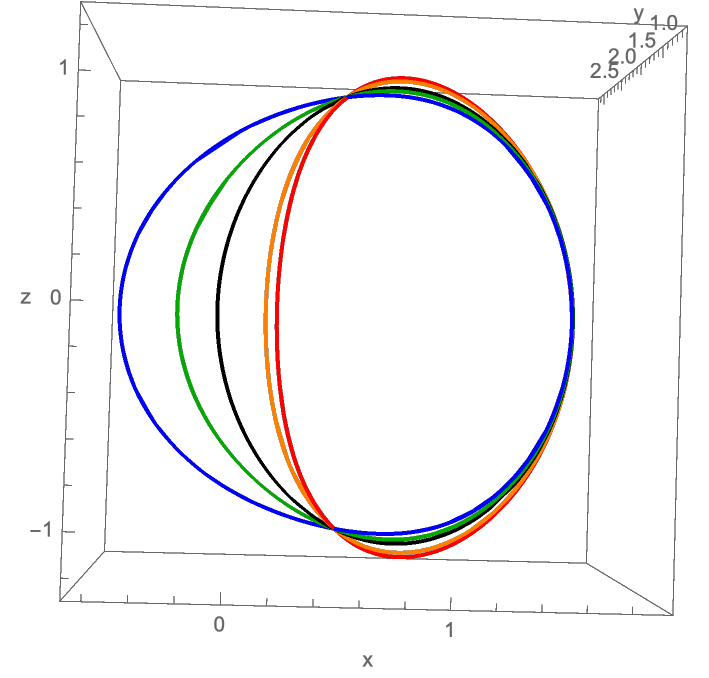}
	\includegraphics[width=8cm, height=5cm]{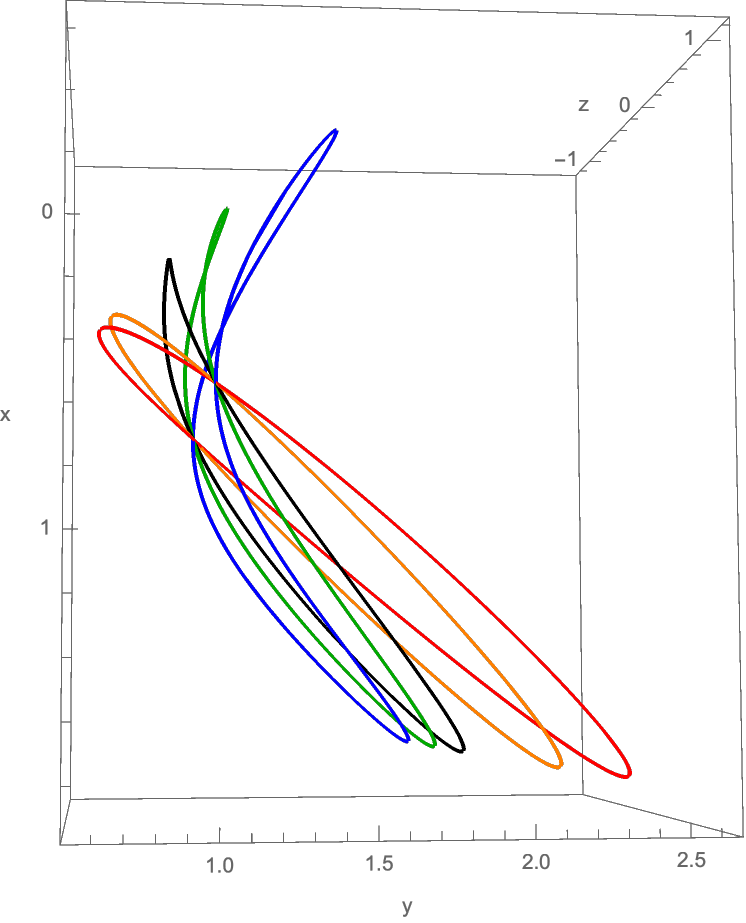}
    \caption{\label{fig:heiswyel}Projections $xz$ (left) and $yx$ (right) of the three-dimensional orbits of the deformed system \eqref{eq:heisdef} with $\eta= -0.5$  (blue), $\eta= -0.25$  (green), $\eta= 0$  (black), $\eta= 1$  (orange) and $\eta= 2$  (red). Initial conditions: $x(0)=0.5$, $y(0)=1$, $z(0)=1$.}
\end{figure}

Closed trajectories for this system are presented in Figure \ref{fig:heiswyel}. Note that trajectories with large positive values of the deformation parameter $\eta$ tend to be contained in a plane, while for large negative values of the deformation parameter $\eta$ their 3-dimensional nature is manifest. 


\subsection{Case AB}

If we want to consider the bi-Hamiltonian deformation, we have to take as Hamiltonian either
\begin{equation}
\mathcal{H}_{1,\eta}=\mathcal{C}_{0,\eta} =\dfrac{e^{-\eta\,z}}{4}(y^2-x^2) ,
\label{eq:h1c0last}
\end{equation}
and the Poisson algebra \eqref{eq:poissonABlambda1}, or equivalently, 
\be
\mathcal{H}_{0,\,\eta}=-\mathcal{C}_{1,\eta}=\dfrac{e^{-\eta\,z}(\eta^2\,(x^2+y^2)+2(e^{\eta\,z}-1)^2)}{4\,\eta^{2}},
\ee
and the Poisson algebra \eqref{eq:poissonABlambda0}. In both cases we recover the same dynamics, explicitly given by 
\be
\begin{array}{l}
\dot{x}=\{x,\mathcal{H}_{0,\eta}\}_{0,\eta} = \{x,\mathcal{H}_{1,\eta}\}_{1,\eta} =\dfrac{y\,e^{-\eta\,z} (-1+e^{2\,\eta\,z}-\eta^{2}\,x^{2})}{2\,\eta},\\
\\
\dot{y}=\{y,\mathcal{H}_{0,\eta}\}_{0,\eta} = \{y,\mathcal{H}_{1,\eta}\}_{1,\eta} =\dfrac{x\,e^{-\eta\,z} (-1+e^{2\,\eta\,z}-\eta^{2}\,y^{2})}{2\,\eta},
\\
\\
\dot{z}=\{z,\mathcal{H}_{0,\eta}\}_{0,\eta} = \{z,\mathcal{H}_{1,\eta}\}_{1,\eta} =-e^{-\eta\,z}x y.
\end{array}\label{e0}
\ee
Again, the non deformed limit leads to the non-deformed equations ~\eqref{RikiType}. A symplectic realization for this bi-Hamiltonian system is given by
\be
\begin{array}{l}
x=\dfrac{1}{\eta}\sqrt{\dfrac{2-4\lambda}{2\lambda-1}}
\sqrt{-(e^{\eta\, p}-1)^2\lambda-2\eta^2\,e^{\eta\,p}|k|}
\sinh (\sqrt{1-2\lambda}\,q ) ,\\
\\
y=\dfrac{2}{\eta\,\sqrt{2\lambda-1}}\sqrt{-e^{\eta\,p}(\eta^2\,|k|+\lambda(\cosh\eta \,p-1))}
\cosh (\sqrt{1-2\lambda}\,q) ,\\
\\
z=p,
\end{array}
\label{eq:sympreallambda59}
\ee
provided that $\lambda \neq 1/2$. That is, \eqref{eq:sympreallambda59} provides a family of parameterizations of the Poisson leaves of the deformed Poisson pencil \eqref{eq:poissonABlambda}.

\begin{figure}[h]
\hspace{-2cm}
    	\includegraphics[width=11cm, height=5cm]{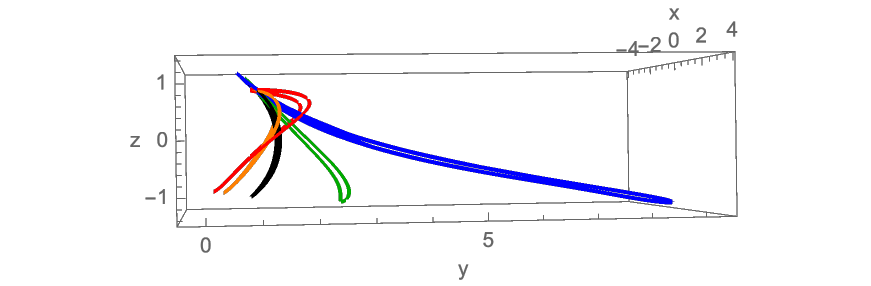}
	\hspace{-2cm}
	\includegraphics[width=11cm, height=5cm]{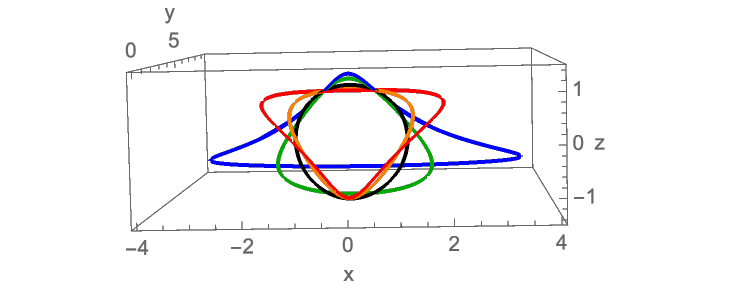}
    \caption{\label{fig:ab} Projections $yz$ (left) and $xz$ (right) of the three-dimensional orbits of the bi-Hamiltonian deformed system \eqref{e0} with $\eta= -2$  (blue), $\eta= -1$  (green), $\eta= 0$  (black), $\eta= 1$  (orange) and $\eta= 2$  (red). Initial conditions: $x(0)=0.5$, $y(0)=1$, $z(0)=1$.}
\end{figure}

It is interesting to comment on some aspects of the deformed dynamics of this system. Firstly, we note that although the deformed system \eqref{e0} is symmetric under the change $x \leftrightarrow y$, this symmetry is not manifest in the numerical integration shown in Figure \ref{fig:ab}. This is due to the fact that the initial conditions for $x$ and $y$ have to be different. In fact, as it is easy to see from the fact that \eqref{eq:h1c0last} is a constant of the motion, there cannot exist closed orbits of the system such that $x(0)=y(0)$, thus preventing the existence of orbits reflecting the $x \leftrightarrow y$ symmetry of the system. Secondly, a closer look at the third equation $\dot z = -e^{-\eta\,z}x y$ shows that if $\eta$ is large and positive, then when $z$ becomes positive, the (closed) trajectories will reach a turning point since $\dot z \to 0$. A similar situation appears when $\eta$ is large and negative, and  $z$ becomes negative. Both of these situations are clearly visible in the red and blue trajectories, respectively, from Figure \ref{fig:ab}.

\section{Coupled systems and cluster variables: the AB case}
\setcounter{equation}{0}

In order to perform a detailed analysis of the coupled systems let us introduce for simplicity the following change of variables for the deformed AB case (\ref{eq:poissonABlambda})
\be
\begin{array}{lll}
x'=e^{-\frac{\eta}{2}z}x, &\quad y'=e^{-\frac{\eta}{2}z}y,&\quad z'=z.\label{CV1}
\end{array}
\ee
Under the transformation (\ref{CV1}) the deformed Poisson brackets read
\be
\begin{array}{lll}
\{x',y'\}_{\lambda,\eta}=\dfrac{2\,\lambda\sinh(\eta\,z')}{\eta}, &
\{x',z'\}_{\lambda,\eta}=(1-2\,\lambda)y', & \{y',z'\}_{\lambda,\eta}=x'.
\end{array}\label{AbDefCV1}
\ee
Applying the transformation (\ref{CV1}) to the AB deformed Casimir function (\ref{eq:CasimirABlambda}) we arrive at the following expression
\be
\mathcal{C}_{\lambda,\eta}=-\dfrac{4\,\lambda (\cosh\,(\eta\,z' )-1)+\eta^{2}((x')^{2}+(2\lambda-1)(y')^2)} 
{4\,\eta^{2}} .
\label{eq:caslambdaeta}
\ee

Finally, we compute the $N=2$ coupled systems for both deformed structures ($\lambda=0, \lambda=1$) and we analyse how the bi-hamiltonian character is broken. From (\ref{eq:caslambdaeta}) and taking $\lambda=1$ we get the Hamiltonian function
\be
\mathcal{H}_{1,\,\eta}\equiv \mathcal{C}_{0,\eta}=\dfrac{1}{4}((y')^2-(x')^2).
\label{H1N}
\ee
Using the Poisson pencil \eqref{AbDefCV1} with $\lambda=1$, the dynamical system in this coordinates is given by
\be
\begin{split}
\dot{x}'&=\dfrac{y'\,\sinh(\eta\, z')}{\eta}, \\
\dot{y}'&=\dfrac{x'\,\sinh(\eta\, z')}{\eta}, \\
\dot{z}'&=-x'\,y',
\end{split}
\label{eq:origdyn}
\ee
and the additional invariant given by the deformed Poisson bracket reads 
\be
\mathcal{C}_{1,\eta}=-\dfrac{4(\cosh(\eta\,z')-1)+\eta^2 ((x')^2+(y')^2)}{4\,\eta^2} .
\ee
In order to compute the $N=2$ coupled system we have to take into account the coproduct~\eqref{defcr1}, which in the new variables~\eqref{CV1} becomes 
\be
\begin{array}{lll}
\Delta_{\eta}(x')=x'_{1}e^{\frac{\eta}{2}z'_{2}}+e^{-\frac{\eta}{2}z'_{1}}x'_{2}, &
\Delta_{\eta}(y')=y'_{1}e^{\frac{\eta}{2}z'_{2}}+e^{-\frac{\eta}{2}z'_{1}}y'_{2}, &
\Delta_{\eta}(z')=z'_{1}+z'_{2}.
\end{array}\label{Cop2}
\ee
 From~\eqref{Cop2} we can define the so called cluster variables \cite{BRcluster} defined through the coproduct
\be
x_{+}=x'_{1}e^{\frac{\eta}{2}z'_{2}}+e^{-\frac{\eta}{2}z'_{1}}x'_{2}, \qquad\qquad
y_{+}=y'_{1}e^{\frac{\eta}{2}z'_{2}}+e^{-\frac{\eta}{2}z'_{1}}y'_{2}, \qquad\qquad
z_{+}=z'_{1}+z'_{2}.
\label{CV2}
 \ee
Using the coproduct \eqref{Cop2} of the Hamiltonian ~\eqref{H1N}, \emph{i.e.}
 \be
 \mathcal{H}^{(2)}_{1\,\eta}:=\Delta_{\eta}(\mathcal{H}_{1,\,\eta})=\dfrac{1}{4}\left(\Delta_{\eta}\left( (y')^{2} \right)-\Delta_{\eta} \left( (x')^{2} \right) \right) ,
 \ee
as the Hamiltonian for the $N=2$ coupled system in the $\lambda=1$ case, the set of coordinates $\{x_+,y_+,z_+,x_1,y_1,z_1\}$ lead to the following system of coupled ODEs:
\be
 \begin{array}{ll}
 \begin{cases}
 \dot{x}_{+}=\dfrac{y_{+}\,\sinh(\eta\, z_{+})}{\eta}\\
 \\
\dot{y}_{+}=\dfrac{x_{+}\,\sinh(\eta\, z_{+})}{\eta}\\
\\
\dot{z}_{+}=-x_{+}\,y_{+}
 \end{cases}
 \hspace{0.5 cm}
 \begin{cases}
 \dot{x}'_{1}=\dfrac{\eta\, y'_{1}}{4}(y_{+}^{2}-x_{+}^{2})+\dfrac{e^{\frac{\eta}{2}(z_{+}-z'_{1})}} {4\,\eta}(4y_{+}\,\sinh(\eta\,z'_{1})+\eta^{2}\,y'_{1}(x_{+}x'_{1}-y_{+}y'_{1}))\\
\\
\dot{y}'_{1}=\dfrac{\eta\,x'_{1}}{4}(x_{+}^{2}-y_{+}^{2})+\dfrac{e^{\frac{\eta}{2}(z_{+}-z'_{1})}} {4\,\eta}(4x_{+}\,\sinh(\eta\,z'_{1})+\eta^{2}\,x'_{1}(-x_{+}x'_{1}+y_{+}y'_{1}))\\
\\
\dot{z}'_{1}=-\dfrac{1}{2}e^{\frac{\eta}{2}\,(z_{+}-z'_{1})}(x_{+}y'_{1}+y_{+}x'_{1}) .
 \end{cases}
 \end{array}
 \ee
 Therefore we note that the dynamics for the cluster variables $\{x_+,y_+,z_+\}$ is exactly the same as the original system, while the dynamics for the other variables ($\{x_1,y_1,z_1\}$ in this case) is much more involved. It should be noted that this is a general feature of our formalism: cluster variables give a set of coordinates (canonically defined by means of the coproduct of the original variables) in which we recover a subsystem with identical dynamics to the original one. However, the ``internal dynamics'' of the coupled system is much more intricate, and this is reflected in the evolution for the rest of the coordinates we choose, which involves the six variables.

A similar construction can be performed for the $\lambda=0$ case. From (\ref{AbDefCV1}) we have the following Poisson-Lie brackets
 \be
\begin{array}{lll}
\{x',y'\}=0, & \{x',z'\}=y', & \{y',z'\}=x'.
\end{array}
\ee
In this case the Hamiltonian function is given by
 \be
 \mathcal{H}_{0,\,\eta}\equiv-\mathcal{C}_{1,\eta}=\dfrac{4(\cosh (\eta\,z' )-1)+\eta^2 \left( (x')^2+(y')^2 \right)}{4\,\eta^2} ,
 \label{eq:H2N}
 \ee
and the additional invariant for the Poisson-Lie algebra is
\be
\mathcal{C}_{0,\eta}=\dfrac{1}{4}\left( (x')^2-(y')^2 \right).
\ee
The equations for the one-copy system are the same as in the $\lambda=1$ case \eqref{eq:origdyn}, \emph{i.e.}
\be
\begin{split}
\dot{x}'&=\dfrac{y'\,\sinh(\eta\, z')}{\eta}, \\
\dot{y}'&=\dfrac{x'\,\sinh(\eta\, z')}{\eta}, \\
\dot{z}'&=-x'\,y'.
\end{split}
\ee
As we already know, the one-copy system is bi-Hamiltonian. This will clearly change when we consider the $N=2$ coupled system defined by the deformed coproduct (\ref{Cop2}) of the Hamiltonian function \eqref{eq:H2N}, namely
\be
 \mathcal{H}^{(2)}_{0\,\eta}=\Delta_{\eta} (\mathcal{H}_{0,\,\eta})=\dfrac{4(\cosh(\eta\,\Delta_{\eta} (z'))-1))+\eta^2 (\Delta_{\eta} (x'^2)+\Delta_{\eta} (y'^2))} 
 {4\,\eta^2} .
\ee

Using again the set of coordinates $\{x_+,y_+,z_+,x_1,y_1,z_1\}$, the dynamical system reads
\be
 \begin{array}{ll}
 \begin{cases}
 \dot{x}_{+}=\dfrac{y_{+}\,\sinh(\eta\, z_{+})}{\eta}, \\
 \\
\dot{y}_{+}=\dfrac{x_{+}\,\sinh(\eta\, z_{+})}{\eta},\\
\\
\dot{z}_{+}=-x_{+}\,y_{+} ,
 \end{cases}
 \hspace{0.5 cm}
 \begin{cases}
 \dot{x}'_{1}=\dfrac{y'_{1}}{4\,\eta}\left(
 4\,	\sinh(\eta\,z_{+})+e^{\frac{\eta}{2}(z_{+}-z'_{1})}\,\eta^{2}(x_{+}x'_{1}+y_{+}y_{1})-\eta^{2}(x_{+}^{2}+y_{+}^{2})
 \right) ,
 \\
\\
\dot{y}'_{1}=\dfrac{x_{1}}{4\,\eta}\left(
 4\,	\sinh(\eta\,z_{+})+e^{\frac{\eta}{2}(z_{+}-z'_{1})}\,\eta^{2}(x_{+}x'_{1}+y_{+}y'_{1})-\eta^{2}(x_{+}^{2}+y_{+}^{2})
 \right) ,
\\
\\
\dot{z}'_{1}=-\dfrac{1}{2}e^{\frac{\eta}{2}\,(z_{+}-z'_{1})}(x_{+}y'_{1}+y_{+}x'_{1}) .
 \end{cases}
 \end{array}
 \ee
As we can straightforwardly see, whilst the subsystem defined by $\{x_+,y_+, z_+\}$ is common for the two $\lambda$-cases (as it should be, since this is a direct consequence of the definition of the cluster variables), the equations for the variables $\{{x}_{1},{y}_{1},{z}_{1}\}$ are different for $\lambda=0$ and $\lambda=1$. Therefore, the coupling procedure defined by means of the Poisson-Lie group structure breaks the initial bi-Hamiltonian character of the system.

\section{Concluding remarks}
\setcounter{equation}{0}

Given any integrable Hamiltonian dynamical system defined in terms of a Lie-Poisson algebra and a Hamiltonian function, the formalism presented in this paper provides a systematic and constructive method to obtain certain integrable deformations of such system that can be generalized to nontrivially coupled versions of it. In particular,  for each Lie bialgebra structure $(\mathfrak g, \delta)$ of the Lie algebra underlying the Poisson-Lie bracket of the initial integrable system, a different integrable deformation can be constructed. The deformed system will be Hamiltonian with respect to a Poisson-Lie structure on the dual Lie group $G^\ast$ whose Lie algebra $\mathfrak g^\ast$ is defined by dualizing the Lie bialgebra map $\delta$, and constants of the motion for the deformed system will be given by the deformed Casimir functions of such Poisson-Lie structure. In this way, the dynamical variables for the deformed system are just the local coordinates for the dual Lie group $G^\ast$. Finally, the coupled system and its constants of the motion will be obtained by making use of the fact that the group multiplication on $G^\ast$ is a Poisson map for the Poisson-Lie structure on $G^\ast$ that defines the deformation. 

Some instances of integrable cases of the Rikitake family of dynamical systems have been considered in order to illustrate this completely generic framework and to emphasise the role of Lie bialgebras and Poisson-Lie groups in dynamical systems theory. In particular, the A system is defined onto the Lie-Poisson (1+1) Poincar\'e algebra, whose complete classification of Lie bialgebra structures is well-known and consists of six different non-trivial cases~\cite{Gomez2000}. Among them, we have considered two representative cases whose dual Lie groups $G^\ast$ are the `book' group and the Heisenberg-Weyl group. A Poisson-Lie structure deforming the Lie-Poisson Poincar\'e algebra can be constructed on each of these two Lie groups, and the two associated integrable deformations of the A system have been explicitly constructed and studied. We stress that, in general, the classification of Lie bialgebra structures (and, therefore, of dual Poisson-Lie group brackets) of a given Lie algebra $\mathfrak g$ is by no means a trivial problem, which is only fully solved for 3D (both complex and real) Lie algebras~\cite{Gomez2000}.

On the other hand, we have also considered the bi-Hamiltonian Rikitake system B in order to study the possibility of obtaining bi-Hamiltonian deformations under the abovementioned framework. This implies the existence of a common cocommutator map $\delta$ for the two Lie algebra structures associated to the bi-Hamiltonian structure of the B system. In this case the answer to this question turns out to be negative, and shows that the preservation of a bi-Hamiltonian structure for Poisson-Lie deformations imposes strong constraints on the formalism. Nevertheless, the particular Rikitake case AB has been shown to admit such a common Lie bialgebra map $\delta$ for the two underlying Lie algebras, and the bi-Hamiltonian deformation can be explicitly constructed. 

We have also provided some numerical simulations for the dynamics of the three deformed Rikitake systems. As expected, all of them present closed trajectories due to their integrability properties, and in general we can see that for negative values of the deformation parameter $\eta$ the difference between the orbits of the deformed and undeformed systems are larger than for positive values of $\eta$. This is quite natural since the deformations do not present the $\eta\to -\eta$ symmetry.

Finally, case AB has been used to exemplify the construction of deformed coupled Rikitalke systems. We have studied the dynamics of such coupled systems, which becomes much more transparent by making use of the so-called `cluster variables'. These are collective variables defined as the deformed coproduct map ({\em i.e.} the group multiplication law on $G^\ast$ for the local coordinates) and the dynamics of the coupled system is such that these cluster variables evolve under the same equations as the dynamical variables of the uncoupled deformed system. At this point it is worth stressing that the construction of coupled systems can be generalized to an arbitrary number $N$ of copies~\cite{BR, BRcluster} just by considering the $N$-th coproduct map $\Delta^{N}$ ({\em i.e.} multiplication of $N$ elements of the $G^\ast$ group) and in that case the global collective variables defined by $\Delta^{N}$ will again reproduce the dynamics of a single deformed system. This generalization to $N$ coupled copies is ensured by construction due to the underlying group structure, and shows that systems obtained through Poisson-Lie deformations constitute an exceptional class of nonlinear dynamical systems.


\section*{Acknowledgements}

This work has been partially supported by Agencia Estatal de Investigaci\'on (Spain) under grant PID2019-106802GB-I00/AEI/10.13039/501100011033. I. Gutierrez-Sagredo thanks Universidad de La Laguna where part of the work has been done for the hospitality and support.


\end{document}